\documentclass[12pt]{amsart}

\usepackage{graphicx}
\usepackage{subfigure}
\usepackage{multirow}
\usepackage{amsfonts}
\usepackage[a4paper,left=2.4cm,right=2.3cm,top=2cm,bottom=1.8cm]{geometry}%
\usepackage{amsmath,amssymb,mathrsfs}
\usepackage{amsthm}
\usepackage{textcomp}
\usepackage{color}

\usepackage{hyperref}
\usepackage{float}
\usepackage{epsfig}
\usepackage{epstopdf}
\usepackage{array,booktabs}
\usepackage{tabularx}
\usepackage[pagewise]{lineno} 
\usepackage{comment} 
\usepackage{bm}
\usepackage{algorithm}
\usepackage{algpseudocode}
\usepackage{listings}
\newtheorem{proposition}{Proposition}
\newtheorem{lemma}{Lemma}
\newtheorem{corollary}{Corollary}

\newtheorem{remark}{Remark}

\title{Convolutional optimization with convex kernel and power lift}
\author{Zhipeng Lu}
\address{Shenzhen MSU-BIT University, \& Guangdong Laboratory of Machine Perception and Intelligent Computing, Shenzhen, Guangdong 518172, China}

\email{zhipeng.lu@hotmail.com}


\begin{document}

\maketitle

\section{Introduction}
In this work, we focus on establishing the foundational paradigm of a novel optimization theory based on convolution with convex kernels. Our goal is to devise a morally deterministic model of locating the global optima of an arbitrary function, which is distinguished from most commonly used statistical models. Limited preliminary numerical results are provided to test the efficiency of some specific algorithms derived from our paradigm, which we hope to stimulate further practical interest.

\section{Theoretical basics}
Our paradigms of optimization theory are motivated by the following basic mathematical facts. 
\subsection{Convolution preserving convexity}\label{sec-convex}
Our optimization paradigm is based on a simple observation on the preservation of convexity by convolution. We start with the univariate case.

\begin{lemma}\label{lem-1-dim convex}
    Suppose $f(x)\geq 0$ is continuous and $g(x)$ is convex. Then their convolution $g\ast f(x)$ is convex. 
\end{lemma}
\begin{proof}
    For any $\theta_1,\theta_2\in\mathbb{R}$, write 
    \begin{align*}
        &g\ast f(\theta_1)+g\ast f(\theta_2)-2g\ast f\left(\frac{\theta_1+\theta_2}{2}\right)\\
        =&\int_{\mathbb{R}}g(\theta_1-t)f(t)dt+\int_{\mathbb{R}}g(\theta_2-t)f(t)dt-2\int_{\mathbb{R}}g\left(\frac{\theta_1+\theta_2}{2}-t\right)f(t)dt\\
        =&\int_{\mathbb{R}}\left(g(\theta_1-t)+g(\theta_2-t)-2g\left(\frac{(\theta_1-t)+(\theta_2-t)}{2}\right)\right)f(t)dt.
    \end{align*}
    If $g$ is convex, $g(\theta_1-t)+g(\theta_2-t)-2g\left(\frac{(\theta_1-t)+(\theta_2-t)}{2}\right)\geq0$ so that the last integral is nonnegative provided $f\geq0$. Thus, $g\ast f$ is convex.
\end{proof}
When $g$ has a continuous second derivative, the proof becomes more concise notionally. Note that $(g\ast f)'=g'\ast f$ and $(g\ast f)''=g''\ast f$. Hence, if $g$ is convex with a continuous second derivative, then $g''\geq 0$ so that $g''\ast f\geq 0$ given that $f\geq 0$. That is, the convolution $g\ast f$ is convex. . In higher dimensions, similarly by linearity $\frac{\partial(g\ast f)}{\partial x_i}=\frac{\partial g}{\partial x_i}\ast f$ and $\frac{\partial^2 g\ast f}{\partial x_i\partial x_j}=\frac{\partial^2 g}{\partial x_i\partial x_j}\ast f$. Also note that convexity in $\mathbb{R}^n$ is equivalent to convexity along any line, i.e. $g(x_0+tv)$ is convex in $t$ for any $x_0, v\in\mathbb{R}^n$. In other words, since $\frac{d^2g(x_0+tv)}{dt^2}=v^T(\nabla^2g)v$ with $\nabla^2g$ the Hessian of $g$, convexity of $g$ means $\nabla^2 g$ is positive semi-definite. Hence again by linearity, for any $x_0,v\in\mathbb{R}^n$, 
\[\frac{d^2 (g(x_0+tv)\ast f)}{dt^2}=(v^T\nabla^2 gv)\ast f,\]
which is non-negative if $f\geq 0$ and $g$ is convex. Thus, without further digestion, for general convex functions in higher dimensions, we also have

\begin{corollary}\label{cor-convex}
    Suppose $f\geq 0$ is continuous and $g$ is convex on $\mathbb{R}^n$ for any $n\geq 1$. Then $g\ast f$ is convex.
\end{corollary}

In practice, if $f$ is not nonnegative, we may increase it by a large enough constant $C>0$ so that $f+C\geq 0$; the continuity of $f$ is necessary as in most algorithms in which discrete sampling is utilized following suitable probability distributions. As we will see, whether $f$ has a unique maximum or not is crucial in our paradigm especially in the process of ``integral concentration".   

\subsection{Critical point approximation}\label{sec-critical}
 The next key principle of our paradigm is to approach the maximum point of $f$ by those of $g_n\ast f$ with a series of convex functions $g_n$. The following observation on critical points of univariate convolutions is key to our construction. 
 
\begin{lemma}\label{lem-critical point}
    Suppose $g$ has a continuous first derivative $|g'(x)|\leq 1$ and $f\geq 0$. If there exist $\delta, \epsilon>0$ such that $|g'(x)|>\frac{1}{2}$ for $|x|\geq\delta$ and $\int_{|x-a|\leq\epsilon}fdx\geq \frac{2}{3}\int_{\mathbb{R}}fdx$ for some $a$, then any potential stationary point of $g\ast f$ lies in $(a-\delta-\epsilon, a+\delta+\epsilon)$.
\end{lemma}
\begin{proof}
    We decompose the integral of $g\ast f$ into two parts ($B_\epsilon(0)=(-\epsilon, \epsilon)$):
    \begin{align*}
        (g\ast f)'(\theta)&=\int_{\mathbb{R}}g'(\theta-a-x)f(a+x)dx\\
        &=\int_{B_\epsilon(0)}g'(\theta-a-x)f(a+x)dx+ \int_{\mathbb{R}\smallsetminus
        B_\epsilon(0)}g'(\theta-a-x)f(a+x)dx.
    \end{align*}
    The second part is easily estimated as
    \[\left|\int_{\mathbb{R}\smallsetminus
        B_\epsilon(0)}g'(\theta-a-x)f(a+x)dx\right|\leq \int_{\mathbb{R}\smallsetminus
        B_\epsilon(0)}f(a+x)dx<\frac{1}{3}\int_{\mathbb{R}}f(x)dx.\]
    If $|\theta-a|\geq \delta+\epsilon$, then $|\theta-a-x|\geq\delta, \forall x\in B_\epsilon(0)$ so that
    \[\int_{B_\epsilon(0)}g'(\theta-a-x)f(a+x)dx\begin{cases}
        >\frac{1}{2}\int_{B_\epsilon(0)}f(a+x)dx>\frac{1}{3}\int_{\mathbb{R}}f(x)dx,\ \theta-a\geq\delta+\epsilon,\\
        <-\frac{1}{2}\int_{B_\epsilon(0)}f(a+x)dx<-\frac{1}{3}\int_{\mathbb{R}}f(x)dx, a-\theta\leq\delta+\epsilon. 
    \end{cases}\]
    Thus if $|\theta-a|\geq \delta+\epsilon$, $(g\ast f)'(\theta)\neq 0$.
\end{proof}

In higher dimensions, since all the first partial derivatives vanish at a critical point, if the conditions of Lemma \ref{lem-critical point} suit for $\frac{\partial g}{\partial x_i}$ and $f$, the same conclusion holds. 

\begin{corollary}\label{cor-higher dim crit. pt}
    Suppose $g$ has continuous partial derivatives $|\partial g/\partial x_i|\leq 1$ and $f\geq 0$ is continuous. If for each $i$, there exists $\delta_i>0$ such that $|\partial g/\partial x_i|>\frac{1}{2}$ for $|x_i|\geq\delta_i$ and $\int_{B_\epsilon(a)}fdx\geq\frac{2}{3}\int_{\mathbb{R}^n}fdx$ for some $\epsilon>0$ and $a=(a_1,\dots,a_n)\in\mathbb{R}^n$, then any potential stationary point $\theta=(\theta_1,\dots,\theta_n)$ of $g\ast f$ satisfies $|\theta_i-a_i|<\delta_i+\epsilon$ for each $i$.
\end{corollary}

\subsection{Construction of convex kernel}\label{sec-construction}

We begin with the one-dimensional case. As indicated by Lemma \ref{lem-critical point}, the convolution kernel $g$ should behave like a ``quasilinear" function. It is also supposed to be convex to avoid local optima traps as conceived in the section \ref{sec-convex}. Choices are still abundant in these regards. However, in practice, since computing convolution is already expensive, our choice had better be fixated on ``low degree" quasilinear functions. This leads us to design $g$ with only linear and quadratic fractions.

To make it even more economical, let $g$ be even with linear fractions $|x\pm\delta|+h$ outside $|x|\leq\delta$ for some border hight $h$; inside $|x|\leq\delta$, we use a quadratic fraction $ax^2+b$ to smooth out the boundary. On the border, continuity requires the latter to reach height $h$, while differentiability requires $2ax=\pm1$. Then we have the simplest system
\begin{equation}\label{eq-1-dim border eq}
   \begin{cases}
       a\delta^2+b=h\\
       2a\delta=1,
   \end{cases}
\end{equation}
which is solved by $a=\frac{1}{2\delta}$ and $b=h-\frac{\delta}{2}$. Thus, we simply define for any $\delta>0$

\begin{equation}\label{eq-1-dim g}
    g(x)=\begin{cases}
        x-\delta+h,\quad x>\delta,\\
        \frac{1}{2\delta}x^2-\frac{\delta}{2}+h, \quad -\delta\leq x\leq\delta,\\
        -x+\delta+h, \quad x<-\delta.
    \end{cases}
\end{equation}

Note that $g$ behaves like a smoothed absolute value function, whose convexity is self-explained. The one-dimensional construction is instructive enough for any higher dimension: meld a cone with a paraboloid along a sphere! Without further digestion, we spout the definition for $\mathbb{R}^n$ as

\begin{equation}\label{eq-higher g}
    g(x)=\begin{cases}
        \|x\|-\delta+h=\sqrt{x_1^2+\cdots+x_n^2}-\delta+h,\quad \|x\|>\delta,\\
        \frac{1}{2\delta}\|x\|^2-\frac{\delta}{2}+h=\frac{1}{2\delta}(x_1^2+\cdots+x_n^2)-\frac{\delta}{2}+h, \quad \|x\|\leq\delta.
    \end{cases}
\end{equation}

If necessary, one may define $g$ as melds of variant cones with general paraboloids along suitable ellipsoids. However, such variant seems not profitable for our paradigm.

 \subsection{Integral concentration}\label{sec-lift}
 Again, as indicated by Lemma \ref{lem-critical point}, the last key point of our paradigm is to rescale $f$ so that its integral is concentrated around its optimum. Our simplest idea is to increase the gap between smaller and larger values while not shifting the variable, which suggests the use of composition functions $\varphi\circ f$ with monotonically increasing functions $\varphi$. Moreover, with respect to the direction of inequality required by Lemma \ref{lem-critical point}, we had better use convex $\varphi$, resorting to Jensen's inequality: \[\frac{1}{b-a}\int_a^b\varphi(f(x))dx\geq\varphi\left(\frac{1}{b-a}\int_a^b fdx\right).\] 

Concerned with computational cost, a convenient choice is $\varphi(x)=x^N$ (for large $N$). In this case, we may derive an efficient bound for the super parameter $N$ as follows.

\begin{lemma}\label{lem-x^N}
    Suppose $f\geq 0$ is continuous with a unique maximum at $a$ and compactly supported in $\mathbb{R}^n$. Then for any $\epsilon>0$ such that the average of $f$ in $B_\epsilon(a)$ is strictly larger than the essential supremum of $f$ outside $B_\epsilon(a)$, say $1+\rho_f(\epsilon)$ multiple of the latter, there exists $N=O\left(\frac{n}{\rho_f(\epsilon)}\ln\left(\frac{1}{\epsilon}\right)\right)$ such that $$\int_{B_\epsilon(a)}\varphi(f(x))dx\geq\frac{2}{3}\int_{\mathbb{R}^n}\varphi(f(x))dx$$
\end{lemma}
\begin{proof}
    Suppose $f$ is supported on a compact set $C\subset \mathbb{R}^n$. Then by Jensen's inequality
    \[\frac{\int_{B_\epsilon(a)}f^Ndx}{\int_{C\smallsetminus B_\epsilon(a)}f^Ndx}\geq \frac{|B_\epsilon(a)|\left(\frac{1}{|B_\epsilon(a)|}\int_{B_\epsilon(a)}fdx\right)^N}{\int_{C\smallsetminus B_\epsilon(a)}f^Ndx}=|B_\epsilon(a)|\left(\frac{\frac{1}{|B_\epsilon(a)|}\int_{B_\epsilon(a)}fdx}{\left(\int_{C\smallsetminus B_\epsilon(a)}f^Ndx\right)^{1/N}}\right)^N.\]
    Note that $\|f\|_{N'}\leq\|f\|_N\leq|C|^{\frac{1}{N}-\frac{1}{N'}}\|f\|_{N'}$ for any $1<N<N'$, and the denominator $\|f|_{C\smallsetminus B_\epsilon(a)}\|_N$ tends to $\|f|_{C\smallsetminus B_\epsilon(a)}\|_\infty\leq (1+\rho_f(\epsilon))\frac{1}{|B_\epsilon(a)|}\int_{B_\epsilon(a)}fdx$ by assumption. In addition, $|B_\epsilon(a)|=O(\epsilon^n)$. Hence, to bound the above ratio away from 0, we need ($|C)|$ is supposed to be negligible)
    $$N=O\left(\frac{\ln(1/\epsilon^n)}{\ln(1+\rho(\epsilon)}\right)=O\left(\frac{n}{\rho_f(\epsilon)}\ln\left(\frac{1}{\epsilon}\right)\right).$$
\end{proof}

We may also use the exponential function $e^{Nx}$ to rescale target functions as how softmax works, but it seems easy to blow up the integral calculation. Compared to our naive baseline setup with simple power functions, one may choose better rescaling functions suitable for the value distributions of special target functions.

\subsection{Restriction of convex functions}\label{sec-restriction}

This section provides a slightly deeper insight into convexity, which inspires us to propose a more delicate optimization method as in section \ref{sec-zigzag}.

We begin with digestion of the two-dimensional case. Suppose $f(x,y)$ is a (bounded) convex function with continuous second derivatives. For any $x$, define $a(x)=\arg\mathrm{max}_yf(x,y)$ and $F(x)=f(x,a(x))$. Then we claim 

\begin{lemma}\label{lem-2-dim restriction}
    The function $a(x)$ is differentiable and the restriction $F(x)$ is also convex.
\end{lemma}
\begin{proof}
    Noting that $\forall x, \frac{\partial f(x,a(x))}{\partial y}=0$, $a(x)$ is clearly continuous with respect to $\frac{\partial f}{\partial y}$. The ordinary differential equation allows for a local differentiable extension of the solution to $a(x)$. More specifically, by differentiating it again, we get (assuming $\frac{\partial^2f}{\partial y^2}\neq 0$.) 
    \[\frac{\partial^2f}{\partial x\partial y}+\frac{\partial^2f}{\partial y^2}a'(x)=0, \ \mathrm{i.e.}, \ a'(x)=-\frac{\frac{\partial^2f}{\partial x\partial y}}{\frac{\partial^2f}{\partial y^2}}.\]

    To check convexity, we just calculate differentials: $F'(x)=\frac{\partial{f}}{\partial x}+\frac{\partial f}{\partial y}a'(x)=\frac{\partial{f(x,a(x))}}{\partial x}$, and 
    \begin{align*}
        F''(x)&=\frac{\partial^2f}{\partial x^2}+2\frac{\partial^2f}{\partial x\partial y}a'(x)+\frac{\partial^2f}{\partial y^2}a'(x)^2\\
        &=(1,a'(x))\begin{pmatrix}
            \frac{\partial^2f}{\partial x^2}&\frac{\partial^2f}{\partial x\partial y}\\
            \frac{\partial^2f}{\partial x\partial y}&\frac{\partial^2f}{\partial y^2}
        \end{pmatrix}\begin{pmatrix}
            1\\
            a'(x),
        \end{pmatrix}
    \end{align*}
    where the matrix is just the Hessian of $f$. If $f$ is convex, its Hessian is semi-positive definite, so that $F''(x)\geq 0, \forall x$. Thus, $F(x)$ must be convex.
\end{proof}

Without further digestion, a similar argument can prove the following.

\begin{corollary}\label{cor-higher-dim restriction}
    For any convex $f$ with continuous second derivatives in $\mathbb{R}^n$, the restriction $f(x_1, \arg\max_{x_2,\dots,x_n}f(x_1,\dots,x_n))$ is also convex.
\end{corollary}

It can be similarly generalized to the restriction of any variables. The above observation serves as the theoretical foundation of a convolutional optimization scheme involving only one-dimensional sampling of function values, which relies on a local-trap-breaking appliance of convolution with convex functions as follows. Suppose $f$ does not attain global optima at $a\in\mathbb{R}^n$. Then for the direction $v$ from $a$ and any convex function $g$ in $\mathbb{R}$, we define the following \textit{directional convolution}:

\begin{equation}\label{eq-directional convolution}
    g\ast_{(a,v)}f(\theta)=\int_\mathbb{R}g(\theta-t)f(a+tv)dt,
\end{equation}
which is convex for $f\geq0$ by Lemma \ref{lem-1-dim convex} and differentiable provided $f$ (or $g$) is. Furthermore, if $g$ satisfies the condition of Lemma \ref{lem-critical point}, say as in the form of the quasi-linear function as of (\ref{eq-1-dim g}), $(g\ast_{(a,v)}f)'(0)\neq 0$ unless $a$ is the global optimum. Specifically, we calculate 
\begin{align}    (g\ast_{(a,v)}f)'(\theta)&=\int_\mathbb{R}g'(\theta-t)f(a+tv)dt=\int_\mathbb{R}g(t)\frac{df(a+(\theta-t)v)}{d\theta}dt\notag\\
    &=\int_\mathbb{R}g(t)(v\cdot\mathrm{grad}f)(a+(\theta-t)v)dt\notag\\
    &=v\cdot\int_\mathbb{R}g(t)(\mathrm{grad}f)(a+(\theta-t)v)dt\label{eq-G expression},
\end{align}
where the last term is short for a vector of integrals. Over all directions $v$ ($\|v\|=1$), the unique one (provided only one optimum) from $a$ to the global optimum has the steepest slope. Thus, the optimal direction to choose is where $G(v)=(g\ast_{(a,v)}f)'(0)$ is stationary. Note that 
\[
    g'(x)=\begin{cases}
    1,\quad x>\delta,\\
        \frac{x}{\delta},\quad -\delta\leq x\leq\delta,\\
        -1,\quad x<-\delta,
\end{cases}\]
we can calculate ($v=(v_1,\dots,v_n)$)
\begin{align}
    \frac{\partial G(v)}{\partial v_i}&=\int_\mathbb{R}g'(-t)t\frac{\partial f(a+tv)}{\partial x_i}dt\notag\\
    &=\int^{-\delta}_{-\infty} t\frac{\partial f(a+tv)}{\partial x_i}dt-\int_{\delta}^{+\infty} t\frac{\partial f(a+tv)}{\partial x_i}dt-\frac{1}{\delta}\int_{-\delta}^{\delta}t^2\frac{\partial f(a+tv)}{\partial x_i}dt.\label{eq-partial G}
\end{align}
\subsubsection{Restricted direction on unit sphere} To solve for $v$ under the condition $\|v\|=1$, the Lagrange multiplier $L(v,\lambda)=G(v)+\lambda(1-\|v\|)$ must have ($\frac{\partial L}{\partial\lambda}=1-\|v\|=0$)
\begin{equation}\label{eq-Lagrange multiplier}
    \frac{\partial L}{\partial v_i}=\frac{\partial G(v)}{\partial v_i}-\lambda\frac{v_i}{\|v\|}=\frac{\partial G(v)}{\partial v_i}-\lambda v_i=0, \forall i,
\end{equation}
which implies $\sum_i v_i\frac{\partial G(v)}{\partial v_i}=v\cdot\mathrm{grad}G(v)=\lambda\|v\|^2=\lambda$. On the other hand, multiplying $v_i$ and summing up through $i$ on (\ref{eq-partial G}) implies
\begin{align}
    v\cdot\mathrm{grad}G(v)=&\int^{-\delta}_{-\infty} t\frac{df(a+tv)}{dt}dt-\int_{\delta}^{+\infty} t\frac{df(a+tv)}{dt}dt-\frac{1}{\delta}\int_{-\delta}^{\delta}t^2\frac{df(a+tv)}{dt}dt\notag\\
    =&tf(a+tv)\mid_{-\infty}^{-\delta}-\int^{-\delta}_{-\infty} f(a+tv)dt-tf(a+tv)\mid^{+\infty}_{\delta}+\int_{\delta}^{+\infty}f(a+tv)dt\notag\\
    &-\frac{1}{\delta}t^2f(a+tv)\mid_{-\delta}^\delta+\frac{1}{\delta}\int_{-\delta}^{\delta}2tf(a+tv)dt\notag\\
    =&-\delta f(a-\delta v)+\delta f(a+\delta v)-\delta(f(a+\delta v)-f(a-\delta v))\notag\\
    &-G(v)+\frac{1}{\delta}\int_{-\delta}^{\delta}tf(a+tv)dt\notag\\
    =&-G(v)+\frac{1}{\delta}\int_{-\delta}^{\delta}tf(a+tv)dt\label{eq-pde of G}
\end{align}
using integral by parts provided that $f$ is differentiable and compactly supported ($f(\infty)=0$). Altogether by (\ref{eq-Lagrange multiplier}) and (\ref{eq-pde of G}) we get 
\begin{equation}\label{eq-lambda}
    \lambda=-G(v)+\frac{1}{\delta}\int_{-\delta}^{\delta}tf(a+tv)dt.
\end{equation}

Note that (\ref{eq-pde of G}) is a first order partial differential equation, whose homogeneous counterpart is $v\cdot\mathrm{grad}G(v)=-G(v)$. Note that the disturbing term behaves more like a linear function as $\delta$ chosen small. Although we may use characteristic ODE (see 3.2.1 in \cite{Evens}) to analyze the solution of (\ref{eq-pde of G}) more accurately, we will utilize a simpler simulation based on the following. 
\begin{lemma}\label{lem-1st order pde}
    For constant $b,c\in\mathbb{R}$, the general solution to $x\cdot\mathrm{grad}f(x)=-f(x)+b\cdot x+c$ is
    $f=C\|x\|^{-1}+\frac{1}{2}b\cdot x+c$.
\end{lemma}
\begin{proof}
    We may rewrite the homogeneous equation as
    \[x\cdot\frac{\mathrm{grad}f}{f}=x\cdot\mathrm{grad}(\ln|f|)=a,\]
    which has an obvious solution $\ln|f|=\|x\|^{-1}$. Suppose there is another solution $f_1$. Then 
    \[x\cdot\mathrm{grad}(\ln|f|-\ln|f_1|)=x\cdot\mathrm{grad}(\ln\left(\frac{|f|}{|f_1|}\right)=0,\]
    which means the directional derivative of $\ln\left(\frac{|f|}{|f_1|}\right)$ is zero in all directions, i.e., $\ln\left(\frac{|f|}{|f_1|}\right)$ is constant. Thus, the homogeneous equation has general solution $f=c\|x\|^{-1}$.

    Let a special solution to the non-homogeneous equation be $f^*=b'\cdot x+c'$. Then we get 
    \[b'\cdot x=(-b'+b)\cdot x-c'+c,\]
    which implies $c'=c,b'=\frac{b}{2}$.
\end{proof}

According to Lemma \ref{lem-1st order pde}, we simulate $G(v)$ by
\begin{equation}\label{eq-simulation of G}
    G(v)\sim C\|v\|^{-1}+\frac{1}{2\delta}\int_{-\delta}^{\delta}tf(a+tv)dt,
\end{equation}
noting that $\frac{1}{\delta}\int_{-\delta}^{\delta}tf(a)dt=0$. Then by (\ref{eq-lambda}), 
$$\lambda\sim -C\|v\|^{-1}+\frac{1}{2\delta}\int_{-\delta}^{\delta}tf(a+tv)dt$$
and (\ref{eq-Lagrange multiplier}) is simulated 
\[\frac{\partial G(v)}{\partial v_i}\sim -C\frac{v_i}{\|v\|^3}+\frac{1}{2\delta}\int_{-\delta}^{\delta}t^2\frac{\partial f(a+tv)}{\partial x_i}dt\sim -C\frac{v_i}{\|v\|}+\frac{v_i}{2\delta}\int_{-\delta}^{\delta}tf(a+tv)dt,\]
that is ($\|v\|=1$),
\begin{equation}\label{eq-determing v}
    \int_{-\delta}^{\delta}t^2\frac{\partial f(a+tv)}{\partial x_i}dt\sim v_i\int_{-\delta}^{\delta}tf(a+tv)dt.
\end{equation}
By expressing $f(a+tv)=f(a)+t\sum_{j}^nv_j\frac{\partial f(a)}{\partial x_j}+O_f(t^2)$, (\ref{eq-determing v}) implies $$\frac{\partial f(a)}{\partial x_i}\sim v_i\sum_{j}^nv_j\frac{\partial f(a)}{\partial x_j}, \forall i,$$
which happens (not surprisingly) only for $v\sim\mathrm{grad}f$ when $\mathrm{grad}f\neq\boldsymbol{0}$. When $\mathrm{grad}f=\boldsymbol{0}$, say when $a$ is a local optimum of $f$, expressing (\ref{eq-determing v}) to the second order implies 
\[H(f)v=(v^TH(f)v)v,\]
which means that we may take $v$ as the main eigenvector of the Hessian $H(f)$ of $f$.
We summarize it as follows.
\begin{lemma}\label{lem-local-trap-breaking}
    Let $f\geq 0$ be differentiable in $\mathbb{R}^n$ with a unique maximum $a_0$ and $g$ as (\ref{eq-1-dim g}). Then at any point $a\neq a_0$, there is a direction $v\in S^{n-1}$ such that $(g\ast_{(a,v)}f)'(0)\neq 0$. Moreover, $v\sim \mathrm{grad} f(a)\neq\boldsymbol{0}$, or $v$ as the main eigenvector of the Hessian of $f$ at $a$ (when $\mathrm{grad} f(a)=\boldsymbol{0}$) is the direction for $(g\ast_{(a,v)}f)'(0)$ to approach the global optimum. 
\end{lemma}

However, the above result is too demanding for the knowledge of data or resources.

\subsubsection{Non-restricted direction}
Actually, it is not only more convenient to optimize the choice of direction $v$ without any restriction, but also more reasonable, noting that $G(v)$ is not radial (see (\ref{eq-G expression})). Similarly by differentiating $G(v)$ at a stationary $v$, multiplying $v_i$ and summing up through $i$, (\ref{eq-pde of G}) shows
\begin{align*}
    0=v\cdot\mathrm{grad}G(v)=-G(v)+\frac{1}{\delta}\int_{-\delta}^{\delta}tf(a+tv)dt,
\end{align*}
or,
\begin{equation}\label{eq-stationary v}
    G(v)=\frac{1}{\delta}\int_{-\delta}^{\delta}tf(a+tv)dt.
\end{equation}
Now let $H(r)=\frac{1}{\delta}\int_{-\delta}^{\delta}tf(a+trv)dt, \forall r>0$. Then $H(1)=G(v)$ shall attain the maximum, hence we should have $H'(1)=0$, i.e., 
\begin{align}
    0=H'(r)\mid_{r=1}=&\frac{1}{\delta}\int_{-\delta}^{\delta}tf(a+trv)dt\notag\\
    =&\frac{1}{\delta}\int_{-\delta}^{\delta}t^2v\cdot\mathrm{grad}f(a+tv)dt\notag\\
    =&\frac{1}{\delta}\int_{-\delta}^{\delta}t^2\frac{df(a+tv)}{dt}dt\notag\\
    =&\delta(f(a+\delta v)-f(a-\delta v))-\frac{2}{\delta}\int_{-\delta}^{\delta}tf(a+tv)dt,\notag
\end{align}
or
\begin{equation}\label{eq-stationary v for r}
    \frac{2}{\delta^2}\int_{-\delta}^{\delta}tf(a+tv)dt=f(a+\delta v)-f(a-\delta v)
\end{equation}
If we still utilize the simulation of $G$ as (\ref{eq-simulation of G}), altogether with (\ref{eq-stationary v}) and (\ref{eq-stationary v for r}), we get
\[G(v)=\frac{1}{\delta}\int_{-\delta}^{\delta}tf(a+tv)dt\sim C\|v\|^{-1}+\frac{1}{2\delta}\int_{-\delta}^{\delta}tf(a+tv)dt\]
so that 
\begin{equation}\label{eq-radius of v}
    C\|v\|^{-1}\sim\frac{1}{2\delta}\int_{-\delta}^{\delta}tf(a+tv)dt=\frac{\delta}{4}(f(a+\delta v)-f(a-\delta v)),
\end{equation}
and 
\begin{equation}\label{eq-estimate of G}
    G(v)\sim\frac{\delta}{2}(f(a+\delta v)-f(a-\delta v)).
\end{equation}
From (\ref{eq-radius of v}) we may estimate that the least radius of $v$ should be
\[\|v\|\sim\sqrt{\frac{8C}{\delta^2\|\mathrm{grad}(f)\|}},\]
while (\ref{eq-estimate of G}) indicates that we need to find a direction at the above radius such that $f(a+\delta v)-f(a-\delta v)$ attains maximum. In summary,

\begin{proposition}\label{prop-directional convolution}
    The direction $v$ on which the directional convolution $(g\ast_{(a,v)} f)'(0)$ attains maximum is the one on radius $\sqrt{\frac{8C}{\delta^2\|\mathrm{grad}(f)\|}}$ such that $f(a+\delta v)-f(a-\delta v)$ attains maximum.
\end{proposition}

When $\mathrm{grad}(f)=\boldsymbol{0}$, we may choose the radius $r$ such that $B_r(a)$ is contained in the support of $f$. In particular, as we approaches the global maximum of $f$, $\|\mathrm{grad}(f)\|$ tends to zero, so that we should the searching radius gradually.

\section{Models of convolutional optimization}
Now we are ready to propose pragmatic models of our convolutional optimization paradigm. We first recall a basic result on the convergence rate of convex optimization, which serves as the baseline for our analysis.

\begin{lemma}[See Theorem 2.6 of \cite{Non-convex}]\label{lem-conv. rate of convex opt.}
    Let $f$ be a convex function in $\mathbb{R}^n$ satisfying
    \[\frac{\alpha}{2}\|x-y\|^2\leq f(y)-f(x)-\nabla f\cdot(y-x)\leq\frac{\beta}{2}\|x-y\|^2,\]
    for some $\alpha,\beta>0$, and execute the projected gradient descent algorithm with step size $\eta\leq\min\{\epsilon,\frac{1}{\beta}\}$. Then after $T=O\left(\frac{1}{\alpha \eta}\ln\left(\frac{\beta}{\epsilon}\right)\right)$ steps, we can approximate the optimum of $f$ by $\epsilon$ error.
\end{lemma}

The condition of Lemma \ref{lem-conv. rate of convex opt.} is more or less a bound of the Hessian spectrum of $f$ if written in Rayleigh quotients. The result is a bit different from the original version of \cite{Non-convex} since we need to change the step size; see details of the proof therein. 

We remark that $\delta$ and $N$ (or $\epsilon$) as in Lemma \ref{lem-critical point} and \ref{lem-x^N} will always be super parameters.

\subsection{One-dimensional model}\label{sec-1-dim alg}
Let $g_\delta$ be as (\ref{eq-1-dim g}) and $F_{\delta,N}=g_\delta\ast f^N$ for any $f\geq 0$ continuous in $\mathbb{R}$. By Lemma \ref{lem-critical point}, we know that the location where $F_{\delta,N}$ attains the optimal (minimum) tends to that of $f$ as $\delta\rightarrow{0}$ and $N\rightarrow{+\infty}$. Also, since $F_{\delta,N}$ is convex by Lemma \ref{lem-1-dim convex}, the location is guaranteed to be approached in spite of some minor boundary conditions, say by gradient descent. Then we simulate $F_{\delta,N}'$ by
\begin{equation}\label{eq-derivative of 1-dim g}    F_{\delta,N}'(\theta)=\int_{\mathbb{R}}g'_\delta(\theta-t)f^N(t)dt=\int^{\theta-\delta}_{-\infty} f^N(t)dt-\int_{\theta+\delta}^{+\infty} f^N(t)dt+\frac{1}{\delta}\int_{\theta-\delta}^{\theta+\delta}(\theta-t)f^N(t)dt
\end{equation} 
The three integrals on the right hand side of (\ref{eq-derivative of 1-dim g}) can be numerically simulated by Riemann sums, in which we suggest the use of step sizes $\geq\delta$ for the first two, while $\frac{\delta}{M}$ (say $M=10$) for the last one since its integral interval is much finer. Note that it is not necessary to recap the first two integrals at each step of gradient descent but rather only count the contribution of the increment intervals. More specifically, for $\theta_2>\theta_1$,

\begin{align}
    F_{\delta,N}'(\theta_2)-F_{\delta,N}'(\theta_1)&=\int_{\theta_1-\delta}^{\theta_1+\delta} f^N(t)dt+2\int_{\theta_1+\delta}^{\theta_2-\delta} f^N(t)dt+\int_{\theta_2-\delta}^{\theta_2+\delta} f^N(t)dt\notag\\
    &\quad+\frac{1}{\delta}\int_{\theta_1-\delta}^{\theta_1+\delta}(t-\theta_1)f^N(t)dt-\frac{1}{\delta}\int_{\theta_2-\delta}^{\theta_2+\delta}(t-\theta_2)f^N(t)dt\label{eq-large increment}\\
    &=\int_{\theta_1-\delta}^{\theta_2-\delta} f^N(t)dt+\int_{\theta_1+\delta}^{\theta_2+\delta} f^N(t)dt+\frac{1}{\delta}\int_{\theta_1-\delta}^{\theta_1+\delta}(t-\theta_1)f^N(t)dt\notag\\
    &\quad-\frac{1}{\delta}\int_{\theta_2-\delta}^{\theta_2+\delta}(t-\theta_2)f^N(t)dt.\label{eq-small increment}
\end{align}
If we allow large increment of steps in gradient descent, we may only the middle term in the first line and the two terms in the second line of (\ref{eq-large increment}); if we only allow small increment, which will be the case in Algorithm \ref{alg-1-dim} since gradients are large for power lift, we may use the simulation of (\ref{eq-small increment}).

In practice, since we usually collect values of target functions in multiples of some unit $\mu$ (say $\mu=0.01$ or so), we may set the quantifier $\rho_f(\epsilon)$ in Lemma \ref{lem-x^N} to be some multiple of $\mu$ (say $\rho_f(\epsilon)=10\mu=0.1$). Then the power $N$ in Lemma \ref{lem-x^N} may be taken $N=O\left(\frac{1}{\mu}\ln\left(\frac{1}{\epsilon}\right)\right)$ (say $N=\frac{2}{\mu}\ln\left(\frac{1}{\epsilon}\right)\sim80$ for $\epsilon=0.02$). In practice, we found the bound of $N$ usually too generous. Moreover, since we want to approximate the location of the global optimum to the error of $\delta+\epsilon$, we may take the step size of gradient descent $\eta=\min\{\delta,\epsilon\}$. (One may set $\delta=\epsilon$ for ease of use.) We assemble the above configuration as Algorithm \ref{alg-1-dim}. 
\begin{algorithm}
\caption{Algorithm of one-dimensional convolutional optimization}\label{alg-1-dim}
\begin{algorithmic}
\Require $f\geq 0, \delta>0, \mu>0, N=O\left(\frac{1}{\mu}\ln\left(\frac{1}{\delta}\right)\right), T=O\left(\frac{1}{\alpha \delta}\ln\left(\frac{\beta}{\delta}\right)\right)$
\Ensure $|\arg\max F_{\delta,N}-\arg\max f|<=\delta$
\State $Int_1 \gets \delta\sum_{i=1}^\infty f^N(x_0-i\delta)-\delta\sum_{i=1}^\infty f^N(x_0+i\delta)$
\State $Int_2 \gets \frac{\delta}{100}\sum_{i=-10}^{10} if^N\left(x_0+\frac{i\delta}{10}\right)$
\State $Int_3\gets 0$
\State $Int=Int_1-Int_2$
\While{$0\leq t\leq T$}
\State $x_{t+1}\gets x_t-\delta\cdot \mathrm{sign}(Int)$
\State $Int_1\gets \mathrm{sign}(Int)\cdot\frac{\delta}{10}\sum_{i=1}^{10}\left(f^N\left(x_t-\delta+\frac{i\delta}{10}\right)+f^N\left(x_{t+1}+\delta+\frac{i\delta}{10}\right)\right)$
\Comment{``sign" is to check the direction of descent}
\State $Int_3\gets Int_2$
\State $Int_2 \gets \frac{\delta}{100}\sum_{i=-10}^{10} if^N\left(x_{t+1}+\frac{i\delta}{10}\right)$
\State $Int\gets Int-Int_1-Int_2+Int_3$
\Comment{Gradient update according to (\ref{eq-small increment}). Note that descent produces opposite direction increment.}
\EndWhile
\end{algorithmic}
\end{algorithm}

To bound $T$ for the convergence rate based as in Lemma \ref{lem-conv. rate of convex opt.}, we investigate the Hessian spectrum of $F_{\delta,N}$. Simply differentiating the integrals of (\ref{eq-derivative of 1-dim g}), we derive
\begin{align}
    F_{\delta,N}''(\theta)&=f^N(\theta-\delta)+f^N(\theta+\delta)+\frac{1}{\delta}\int_{\theta-\delta}^{\theta+\delta}f^N(t)dt+\frac{\theta}{\delta}\left(f^N(\theta+\delta)-f^N(\theta-\delta)\right)\notag\\
    &\quad -\frac{1}{\delta}\left((\theta+\delta)f^N(\theta+\delta)-(\theta-\delta)f^N(\theta-\delta)\right)\notag\\
   &=\frac{1}{\delta}\int_{\theta-\delta}^{\theta+\delta}f^N(t)dt\label{eq-2nd derivative 1-dim}
\end{align}
If $f$ is continuous, it becomes $2f^{N}(\xi)$ for some $\theta-\delta<\xi<\theta+\delta$. Thus, in the vanilla use of gradient descent, the simplest estimate of the Hessian bounds would be the optima of $f^N$, i.e. $\alpha=(\min f)^N=m^N$ and $\beta=(\max f)^N=M^N$ (for $f>0$). We may rescale $f$ by $\frac{1}{m}f$ to make its minimal $1$ and maximal $\frac{M}{m}$ (while not changing the ratio $\rho(f)$ for the ease of Lemma \ref{lem-x^N}). Then the bound for $T$ would be $T=O\left(\frac{N}{\delta}\ln\left(\frac{M}{m\delta}\right)\right)$. In practice, we would rather use normalized gradient descent (``sign" descent in 1-dim), so that we may only take $T$ to be $T=O\left(\frac{|C|}{\delta}\right)$, where $|C|$ is the volume of the (compact) support $C$ of $f$.

We summarize the above reasoning as follows.

\begin{proposition}\label{prop-1-dim analysis}
    Assume the same conditions of Lemma \ref{lem-critical point} and \ref{lem-critical point} ($\delta=\epsilon=\eta, f\geq 1$ compactly supported) and $m\leq f\leq M$ for some $m,M>0$ with the value gap $\mu>0$. Then by gradient descent as in Lemma \ref{lem-conv. rate of convex opt.} with $N=O\left(\frac{1}{\mu}\ln\left(\frac{1}{\delta}\right)\right)$, after $T=O\left(\frac{1}{\mu\delta}\ln\left(\frac{M}{m\delta}\right)\ln\left(\frac{1}{\delta}\right)\right)$ steps, the location of the maximum of $f$ can be approximated to $\delta$ error. If using a normalized gradient descent as in Algorithm \ref{alg-1-dim}, the optimum of $f$ can be approximated to $\delta$ error in $O\left(\frac{1}{\delta}\right)$ steps.
\end{proposition}

\begin{remark}
    Note that by (\ref{eq-2nd derivative 1-dim}) the gradient increment ($\sim f^N(\xi)\delta$) tends to explode for $N$ large. In this regard, Algorithm \ref{alg-1-dim} suggests the use of the sign of the gradient to simply determine the direction of descent without any augmentation of step length. The bounds on $T$ and $N$ do not rely on any simulation of target functions, but rather on value distributions. Since data are usually distributed within some range, we always assume target functions are compactly supported and in programming, we may set zero value outside given ranges. 
\end{remark}

\subsection{Higher-dimensional model}\label{sec-higher-dim alg}

Similarly, given the construction of (\ref{eq-higher g}), we have 
\[\frac{\partial g(x)}{\partial x_i}=\begin{cases}
        \frac{x_i}{\sqrt{x_1^2+\cdots+x_n^2}},\quad \|x\|>\delta,\\
        \frac{x_i}{\delta}, \quad \|x\|\leq\delta,
    \end{cases}\]
and so ($\theta=(\theta_1,\dots,\theta_n)$)
\begin{align}\label{eq-derivative of higher g}
\frac{\partial F_{\delta,N}(\theta)}{\partial \theta_i}&=\int_{\mathbb{R}^n}\frac{\partial g(\theta-x)}{\partial x_i}f^N(x)dx\notag\\
&=\int_{\mathbb{R}^n\smallsetminus B_\delta(\theta)}\frac{(\theta_i-x_i)f^N(x)}{\sqrt{(\theta_1-x_1)^2+\cdots+(\theta_n-x_n)^2}}dx-\frac{1}{\delta}\int_{B_\delta(\theta)}(\theta_i-x_i)f^N(x)dx\notag\\
&\sim\int_{\mathbb{R}^n}\frac{(\theta_i-x_i)f^N(x)}{\sqrt{(\theta_1-x_1)^2+\cdots+(\theta_n-x_n)^2}}dx-\frac{1}{\delta}\int_{B_\delta(\theta)}(\theta_i-x_i)f^N(x)dx,
\end{align}
noticing that $\int_{B_\delta(\theta)}\frac{(\theta_i-x_i)f^N(x)}{\sqrt{(\theta_1-x_1)^2+\cdots+(\theta_n-x_n)^2}}dx\sim 0$ for $\delta$ small enough. In practice, we may divide $\mathbb{R}^n$ roughly, say with $\frac{C}{10}$ increment ($C$ is the size of compact support of $f$) for the first integral, while taking $\frac{\delta}{10}$ similarly for the second, so that total sampling of the values amount to $O(10^n)$. This bound is inapplicable when $n$ is large. In addition, compared to (\ref{eq-derivative of 1-dim g}), it seems that the above formula is not helpful for devising an increment updating method similar to (\ref{eq-small increment}) to avoid computing integrals over entire domains. Despite these defects, we propose Algorithm \ref{alg-higher dim} for higher-dimensional convolutional optimization and leave the issue for the zigzag variant algorithm based on Algorithm \ref{alg-1-dim} to resolve.

\begin{algorithm}
\caption{Algorithm of higher-dimensional convolutional optimization}\label{alg-higher dim}
\begin{algorithmic}
\Require $f\geq 0$ supported on $[-k,k]^n\subset\mathbb{R}^n$ for some positive integer $k$, $\delta>0, \mu>0, N=O\left(\frac{1}{\mu}\ln\left(\frac{1}{\delta}\right)\right), T=O\left(\frac{1}{\alpha \delta}\ln\left(\frac{\beta}{\delta}\right)\right)$
\Ensure $|(\arg\max F_{\delta,N}-\arg\max f)_i|<=\delta, \forall i$
\State $\forall 1\leq i\leq n, Int_{1i} \gets \frac{1}{10^n}\sum_{j=1}^n\sum_{l_j=-10k}^{10k}\frac{(x_{0i}-l_i/10)f^N(\boldsymbol{l})}{\sqrt{(x_{01}-l_1/10)^2+\cdots+(x_{0n}-l_n/10)^2}}$
\State $\forall 1\leq i\leq n, Int_{2i} \gets \frac{\delta^{n}}{10^{n+1}}\sum_{\|\boldsymbol{i}\|_1\leq10} \boldsymbol{i}_if^N\left(x_0+\frac{\delta}{10}\boldsymbol{i}\right)$
\State $\forall 1\leq i\leq n, Int_i=Int_{2i}-Int_{1i}$
\While{$0\leq t\leq T$}
\State $x_{t+1}\gets x_t-\delta\cdot\left(\mathrm{sign}(Int_1),\dots,\mathrm{sign}(Int_n)\right)$
\State $\forall 1\leq i\leq n, Int_{1i} \gets \frac{1}{10^n}\sum_{j=1}^n\sum_{l_j=-10k}^{10k}\frac{(x_{(t+1)i}-l_i/10)f^N(\boldsymbol{l})}{\sqrt{(x_{(t+1)1}-l_1/10)^2+\cdots+(x_{(t+1)n}-l_n/10)^2}}$
\State $\forall 1\leq i\leq n, Int_{2i} \gets \frac{\delta^{n}}{10^{n+1}}\sum_{\|\boldsymbol{i}\|_1\leq10} \boldsymbol{i}_if^N\left(x_{t+1}+\frac{\delta}{10}\boldsymbol{i}\right)$
\State $\forall 1\leq i\leq n, Int_i=Int_{2i}-Int_{1i}$
\Comment{Gradient update according to (\ref{eq-derivative of higher g}).}
\EndWhile
\end{algorithmic}
\end{algorithm}

Similar to the one-dimensional case, we write 
\begin{align}
\frac{\partial^2 F_{\delta,N}(\theta)}{\partial\theta_i\partial \theta_j}&=\frac{\partial}{\partial\theta_j}\int_{\mathbb{R}^n\smallsetminus B_\delta(\theta)}\frac{(\theta_i-x_i)f^N(x)}{\sqrt{(\theta_1-x_1)^2+\cdots+(\theta_n-x_n)^2}}dx\notag\\
&\quad-\frac{1}{\delta}\frac{\partial}{\partial\theta_j}\int_{B_\delta(\theta)}(\theta_i-x_i)f^N(x)dx.
\end{align}
Although the explicit calculation of Hessian spectrum of $F_{\delta,N}$ using the above formula is inconvenient, without much further digestion, we conclude

\begin{proposition}\label{prop-higher-dim analysis}
    Assume the same conditions of Lemma \ref{lem-critical point} and \ref{lem-critical point} ($\delta=\epsilon=\eta, f\geq 1$ compactly supported) and $m\leq f\leq M$ for some $m,M>0$ with the value gap $\mu>0$. Then by gradient descent as in Lemma \ref{lem-conv. rate of convex opt.} with $N=O\left(\frac{n}{\mu}\ln\left(\frac{1}{\delta}\right)\right)$, after $T=O\left(\frac{n}{\mu\delta}\ln\left(\frac{M}{m\delta}\right)\ln\left(\frac{1}{\delta}\right)\right)$ steps, the location of the maximum of $f$ can be approximated to $\delta$ error. If using a normalized gradient descent as in Algorithm \ref{alg-1-dim}, the optimum of $f$ can be approximated to $\delta$ error in $O\left(\frac{1}{\delta}\right)$ steps.
\end{proposition}

\subsection{One-dimensional ``zigzag" variant algorithms for higher dimensional optimization}\label{sec-zigzag}

When it is inconvenient to sample the values of higher-dimensional functions by convolution with convex kernels, we propose a one-dimensional substitute, which behaves like ``zigzag" walking on hypersurfaces, based on restriction of convex functions as in Section \ref{sec-restriction}. We first propose the zigzag variant Algorithm \ref{alg-zigzag1} to simply reduce the complexity of sampling and the increment calculation issue from Section \ref{sec-higher-dim alg}. 

\begin{algorithm}
\caption{1st zigzag algorithm: zigzag1}\label{alg-zigzag1}
\begin{algorithmic}
\Require $f\geq 0, \delta>0, \mu>0, N=O\left(\frac{1}{\mu}\ln\left(\frac{1}{\delta}\right)\right), T=O\left(\frac{1}{\alpha \delta}\ln\left(\frac{\beta}{\delta}\right)\right), K=O\left(n\right)$
\Ensure $\|\arg\max F_{\delta,N}-\arg\max f\|<=\delta$
\State $\forall 1\leq i\leq n, f_{x_0,i}(t)\gets f(x_0+t\boldsymbol{e}_i)$
\Comment{Define or insert restriction functions}
\State $\forall t, \mathrm{mod}_n(t)\gets 1+(t\mod{n})$
\Comment{Define or insert modulo valued 1 to $n$}
\State $\mathrm{COCP_1}(f,N,T,x_0,i)\gets \text{the output of Algorithm } \ref{alg-1-dim} \text{ on } f_{x_0,i}(t)$ for given $N, T$, initial $x_0$
\Comment{Define or insert Algorithm \ref{alg-1-dim} as a module}
\While{$0\leq t\leq K$}
\State $x_{t+1}\gets\mathrm{COCP}_1(f,N,T,x_t,\mathrm{mod}_n(t))$
\EndWhile
\end{algorithmic}
\end{algorithm}

Note that Algorithm \ref{alg-zigzag1} employs a ``lazy" strategy of choosing directions cyclically. In theory, if a minor direction contributes little to gradient descent, $x_t$ does not move far until it eventually turns to a major direction in a cycle. We may also take other strategies of choosing directions, say, at each step choose the direction to which its partial derivative attains the maximum, which certainly takes more time. Whichever strategy we employ, Proposition \ref{prop-zigzag} indicates that it takes $K=O(n)$ cycles to work.

Unfortunately, Algorithm \ref{alg-zigzag1} may not avoid local traps (see Table \ref{table-zigzag}), for which reason we utilize Proposition \ref{prop-directional convolution} to devise another zigzag algorithm as follows. Note that we sample a direction of zigzag walking for each iteration to reduce complexity, which actually works as well.

\begin{algorithm}
\caption{2nd zigzag algorithm: zigzag2}\label{alg-zigzag2}
\begin{algorithmic}
\Require $f\geq 0, \delta>0, \mu>0, N=O\left(\frac{1}{\mu}\ln\left(\frac{1}{\delta}\right)\right), T=O\left(\frac{1}{\alpha \delta}\ln\left(\frac{\beta}{\delta}\right)\right), K=O(n), \mathcal{N}_{S^{n-1}}$ (uniform distribution on the sphere $S^{n-1}$)
\Ensure $\|\arg\max F_{\delta,N}-\arg\max f\|<=\delta$
\State $\forall x_0,v_0\in\mathbb{R}^n, \|v_0\|=1, r_0>0, f_{x_0,v_0,r_0}(t)\gets f(x_0+tv_0), -r_0\leq t\leq r_0; 0$ otherwise
\Comment{Define or insert restriction functions with given starting point, direction and radius}
\State $r_0\gets \delta$
\State $\mathrm{COCP_1}(f,N,T,x_0,i)\gets \text{the output of Algorithm } \ref{alg-1-dim} \text{ on } f$ for given $N, T$
\Comment{Define or insert Algorithm \ref{alg-1-dim} as a module}
\While{$0\leq t\leq K$}
\State $v_{t+1}=\arg\max_{\|v\|=r_t, v/r_t\sim\mathcal{N}_{S^{n-1}}}|f(x_t+\delta v)-f(x_t-\delta v)|$
\State $x_{t+1}\gets\mathrm{COCP}_1(f_{x_t,v_t,r_t},N,T)$
\State $r_{t+1}=r_{t}+\delta$
\EndWhile
\end{algorithmic}
\end{algorithm}

\begin{proposition}\label{prop-zigzag}
    With the condition of Lemma \ref{lem-critical point}, starting from any point $a_0\in\mathbb{R}^n$, Algorithm \ref{alg-zigzag2} would converge (with probability tending to 1 as sampling of directions approximates the normal or uniform distribution) to the global maximum of $f$ with distance $O(\delta)$, after $K=O(n)$ iterations and $T=O\left(\frac{1}{\alpha(\delta+\epsilon)}\ln\left(\frac{\beta}{\delta+\epsilon}\right)\right)$ steps in each zigzag walking. 
    
\end{proposition}

\begin{proof}  
Clearly, the values $f(a_k)$ are sliding upward (downward) as we optimize along the axes in each step, and by Proposition \ref{prop-directional convolution} we can always jump out of local optimum traps so that eventually $f(a_k)$ tends to the global optimum. More specifically, since $i_k\in\{1,\dots,n\}$, there must be a subsequence $i_{k_l}=j$ for some $j$. Thus, the zigzag optimization can be seen as an optimization of the restriction $g\ast f$ onto $x_j$ and $\arg\max_{x_1,\dots,x_{j-1},x_{j+1},\dots,x_n}g\ast f(x_1,\dots,x_n)$. Note that the convolution with a higher dimensional $g$ degenerates to a directional convolution when restricted, say for the construction of (\ref{eq-1-dim g}) and (\ref{eq-higher g}). By Corollary \ref{cor-higher-dim restriction}, the restriction is convex, so that the global optimum is guaranteed to be obtained by optimization. Moreover, by Lemma \ref{lem-conv. rate of convex opt.} below, the convergence rate of the restriction is $O\left(\frac{1}{\alpha(\delta+\epsilon)}\ln\left(\frac{\beta}{\delta+\epsilon}\right)\right)$. Here, $\alpha$ and $\beta$ may be seen as the smallest and largest eigenvalues of the Hessian of $g\ast f$, which are positive by assumption (the convolution is strictly convex). Note also that we are supposed to take the step size $\sim\delta+\epsilon$ in gradient descent. The multiple of $n$ is for the relaxation of walking on arbitrary axes.

\end{proof}

\section{Computational results de facto}

\subsection{Disclaimer about limited experimental results}
The main body of this work is focused on the device of algorithms from theoretical baselines. We will provide basic experimental verification of Algorithm \ref{alg-1-dim} and \ref{alg-higher dim} on test functions including the Rosenbrock function and some logarithmic potentials. Codes in C language can be found in Appendix \ref{sec-COCP1} and \ref{sec-COCP2}. Additionally, we exhibit some preliminary results of the zigzag algorithms on a bivariate logarithmic potential with C or C++ codes in Appendix \ref{sec-zigzag1} and \ref{sec-zigzag2}. 
\subsection{Summary of results}
Table \ref{sample-table} summarizes our preliminary results on two kinds of common test functions with chosen super parameters $\delta$ (step size) and $N$ (power). Here, $\mathrm{LOG}_1$ denotes the function $\max\{0,-\log((x-0.5)^2+0.00001)-\log((x-1)^2+0.01)\}$ restricted between $-0.2$ and $1.6$ with global max at $0.5$ and local max at $1$; $\mathrm{POLY}_1$ is $\max\{0, -x^6 + 2x^5 - 4x + 3\}$ restricted between $-2$ and $2$ with global max at $-0.726\dots$ and local max at $1.551\dots$; $\mathrm{LOG}_1$ is $\max\{0,-\log((x-0.5)^2+(y-0.5)^2+0.00001)-\log((x+0.5)^2+(y+0.5)^2+0.01)\}$ restricted in $[-1,1]^2$ with global max at $(0.5,0.5)$ and local max at $(-0.5,-0.5)$; $\mathrm{Rb}_2$ is the Rosenbrock function $1.2-100*(y-x^2)^2-(1-x)^2$ restricted in $[-1.5,1.5]^2$ with global max at $(1,1)$. The number $T$ of steps are usually $O(1/\delta)$.

\begin{table}[t]
\caption{Results of $\mathrm{COCP}_1$ and $\mathrm{COCP}_2$}
\label{sample-table}
\vskip -0.3in
\begin{center}
\begin{small}
\begin{sc}
\begin{tabular}{lcccr}
\toprule
Function & $\delta$ & $N$ & $x_0$ & $x_T$\\
\midrule
$\mathrm{LOG}_1$    & 0.01 & 3 & 1.5 & 0.51 \\
$\mathrm{LOG}_1$    & 0.001 & 6 & 0.8 & 0.501 \\
$\mathrm{POLY}_1$   & 0.01 & 15& 0.5 & -0.71\\
$\mathrm{POLY}_1$    & 0.01 & 15 & 1.8 &  -0.71\\
$\mathrm{LOG}_2$     & 0.1 & 3 & (0.3,-0.3) & (0.5,0.5)  \\
$\mathrm{LOG}_2$      & 0.05 & 3 & (0.8,0.9) & (0.5,0.5) \\
$\mathrm{Rb}_2$     & 0.1 & 5 & (0.3,-0.2) &  (1.0,0.9)       \\
$\mathrm{Rb}_2$   & 0.05 & 8 & (1.2,-1.3) & (1.05,0.95)\\
\bottomrule
\end{tabular}
\end{sc}
\end{small}
\end{center}
\vskip 0.2in
\end{table}

Table \ref{table-zigzag} shows the results of Algorithm \ref{alg-zigzag1} and \ref{alg-zigzag2} on $\mathrm{LOG}_2$, the latter of which aligns with Proposition \ref{prop-zigzag} and is bit sensitive to the step number $T$ (inversely proportional $\delta$ or so). The lazy zigzag Algorithm \ref{alg-zigzag1} performs poorly when the initial $x$-value is negative as shown by the last 7 rows of Table \ref{table-zigzag}. Here it suffices to set $N=3$ and $K=10\sim 15$.

\begin{table}[t]
\caption{Results of zigzag1 and zigzag2 on $\mathrm{LOG}_2$}
\label{table-zigzag}
\vskip -0.3in
\begin{center}
\begin{small}
\begin{sc}
\begin{tabular}{lcccccr}
\toprule
Algorithm & $\delta$ & $T$ & $x_0$ & $y_0$ & $x_K$ & $y_K$\\
\midrule
zigzag1 & 0.05 & 50 & 0.0 & 0.0 & 0.50 & 0.55 \\
zigzag2   & 0.05 & 50 & 0.0 & 0.0 & 0.49 & 0.50 \\
zigzag1 & 0.05 & 50 & -0.5 & -0.5 & -0.35 & -0.40 \\
zigzag2   & 0.05 & 50 & -0.5 & -0.5 & 0.49 & 0.48 \\
zigzag1 & 0.02 & 100 & 0.0 & 0.0 & 0.52 & 0.52 \\
zigzag2   & 0.02 & 100 & 0.0 & 0.0 & 0.51 & 0.49 \\
zigzag1 & 0.02 & 200 & -0.5 & -0.5 & 0.46 & 0.52 \\
zigzag2   & 0.02 & 100 & -0.5 & -0.5 & 0.51 & 0.51 \\
zigzag1 & 0.01 & 200 & -0.1 & -0.1 & -0.41 & -0.40 \\
zigzag1 & 0.01 & 200 & -0.4 & -0.6 & -0.41 & -0.40 \\
zigzag1 & 0.01 & 200 & -1.0 & -1.0 & -0.41 & -0.40 \\
zigzag1 & 0.01 & 200 & -0.1 & 0.2 & -0.41 & -0.40 \\
zigzag1 & 0.01 & 200 & 0.1 & -0.9 & 0.43 & 0.44 \\
zigzag1 & 0.02 & 200 & 0.1 & -1.0 & 0.52 & 0.50 \\
zigzag1 & 0.02 & 200 & 1.0 & -1.0 & 0.52 & 0.52 \\
\bottomrule
\end{tabular}
\end{sc}
\end{small}
\end{center}
\vskip -0.1in
\end{table}

\newpage
\appendix
\onecolumn
\section{\texorpdfstring{$\mathrm{COCP}_1$ in C codes (for $\mathbf{LOG}_1$)}{COCP1 in C codes (for LOG1)}}\label{sec-COCP1}
\begin{lstlisting}
    #include <stdio.h>
#include <math.h>

#define N 3 //Power
#define T 200 //Steps
#define DELTA 0.01 //Gradient descent gap

// Function: f(x)=-log((x-0.5)^2+0.00001)-log((x-1)^2+0.01) on [-0.2, 1.6]

double f(double x) {
    if (x < -0.2 || x >1.6) {
        return 0.0; // Outside range
    }
    return pow(-log(pow(x - 0.5, 2) + 0.00001) - log(pow(x - 1, 2) + 0.01), N); 
    // Inside range
}

// Main function
int main() {
    double x_t = 0.8; // Initialize x_0
    double Int1 = 0.0, Int2 = 0.0, Int3 = 0.0, Int = 0.0;

    // Calculate Int1
    for (int i = 1; i < T; ++i) { // Large range to approximate infinity
        double term1 = f(x_t-i * DELTA);
        double term2 = f(x_t+i * DELTA);
        if (term1 != 0.0) Int1 += DELTA * term1;
        if (term2 != 0.0) Int1 -= DELTA * term2;
    }

    // Calculate Int2
    for (int i = -10; i <= 10; ++i) {
        double xi = i * DELTA / 10.0;
        if (xi >= -0.2 && xi <= 1.6) { // Ensure xi is within range
            Int2 += (i * DELTA* f(x_t+xi)) / 100.0;
        }
    }

    Int = Int1 - Int2;

    // Iterative updates
    for (int t = 0; t <= T; ++t) {
        double y=x_t;
        x_t -= DELTA * copysign(1.0, Int);

        // Update Int1
        Int1 = 0.0;
        for (int i = 1; i <= 10; ++i) {
            double xi = x_t + DELTA + i * DELTA/10.0;
            double yi= y - DELTA + i * DELTA/10.0;
            if (xi >= -0.2 && xi <= 1.6) { // Ensure xi is within range
                Int1 += copysign(1.0, Int) * DELTA/10.0 * f(xi) 
                + copysign(1.0, Int) * DELTA/10.0 * f(yi);
            }
        }

        // Update Int3 and Int2
        Int3 = Int2;
        Int2 = 0.0;
        for (int i = -10; i <= 10; ++i) {
            double shifted_x = x_t + i * DELTA / 10.0;
            if (shifted_x >= -0.2 && shifted_x <= 1.6) { 
            // Ensure shifted_x is within range
                Int2 += (i * DELTA* f(shifted_x)) / 100.0;
            }
        }

        // Update Int
        Int = Int - Int1 - Int2 + Int3;

        // Optional Debugging Output
        // printf("Iteration %d: x_t = %f\n", t, x_t);
    }

    // Print the final result
    printf("Final value of x_t: %f\n", x_t);

    return 0;
}

\end{lstlisting}

\section{\texorpdfstring{$\mathrm{COCP}_2$ in C codes (for Rosenbrock)}{COCP2 in C codes (for Rosenbrock)}}\label{sec-COCP2}


\begin{lstlisting}
   #include <stdio.h>
#include <math.h>

#define N 2 // Define a value for N
#define NP 8 // Define a value of power
#define K 1.5 // Example value for K, adjust accordingly
#define DELTA 0.05 // Example value for delta, adjust accordingly
#define DELTA1 0.2 // Differential gap for Int_1

// Define the test function f

double f(double *x) {
    // Example: f(x) =1.5-100(y-x^2)^2-(1-x)^2
    
    double result=-100*(x[1]-x[0]*x[0])*(x[1]-x[0]*x[0])
    -(1-x[0])*(1-x[0])+1.5;
    if (result>0){return pow(result,NP);}
    else{return 0.0;}
}

//Euclidean norm for Int_1

double norm(double *x, double *y) {
    double sum = 0.0;
    for (int i = 0; i < N; i++) {
        sum += (x[i] - y[i]) * (x[i] - y[i]);
    }
    return sqrt(sum);
}

// Calculate the integrand (x_i - y_i)*f(y)/||x - y|| for Int_1

double integrand1(double *x, double *y, int i) {
       int sign = (x[i] - y[i] > 0) ? 1 : -1;
    double norm_val = norm(x, y);
    int j=1-i;
    if (fabs(x[j]-y[j])<DELTA1)
    {return sign*f(y);}
    else {
        return (x[i] - y[i]) * f(y) / norm_val;
         }
}

// Calculate the integrand (x_i - y_i) * f(y) for Int_2

double integrand2(double *x, double *y, int i) {
    
        return (x[i] - y[i]) * f(y);
}

// Integrate (x_i - y_i) * f(y) / ||x - y|| for Int_1

double integrate1(double *x, int i) {
    double result = 0.0;
    
    // Iterate over all combinations of y in [-K, K]^2 (since N=2)
    // For simplicity, we will sample values in a grid within [-K, K]
    
    int steps=(int)(K/DELTA1);
    
    // Number of grid points per dimension
    
    double y[N];

    // Loop through the grid for each dimension (2D)
    for (int i1 = -steps; i1 <= steps; i1++) {
        for (int i2 = -steps; i2 <= steps; i2++) {
            // Populate y with current grid points
            
            y[0] = i1 * DELTA1;
            y[1] = i2 * DELTA1;
                result += integrand1(x, y, i) * DELTA1 * DELTA1; 
                // Multiply by volume element
        }
    }
    return result;
}

//Integrate (x_i - y_i) * f(y) for Int_2

double integrate2(double *x, int i) {
    double result = 0.0;
    for (int i1 = -10; i1 <= 10; i1++) {
        for (int i2 = -10; i2 <= 10; i2++) {
            //Differential gap $\delta/10$
            
            y[0] = x[0]+i1 * DELTA/10.0;
            y[1] = x[1]+i2 * DELTA/10.0;

            // Evaluate the integrand for each dimension
  
                result += integrand2(x, y, i)* DELTA/100.0; 
                // Multiply by volume element
            }
    }
    return result;
}

// Function to compute Int_1 for a given vector x
void compute_Int_1(double *x, double *Int_1) {
    // Initialize Int_1 to zero
    
    for (int i = 0; i < N; i++) {
        Int_1[i] = 0.0;
    // Compute the sum for Int_1
    
        Int_1[i] = integrate1(x, i);  // scale by 10^n
    }
}

// Function to compute Int_2 for a given vector x

void compute_Int_2(double *x, double *Int_2) {
    // Initialize Int_2 to zero
    
    for (int i = 0; i < N; i++) {
        Int_2[i] = 0.0;
    }
    // Compute the sum for Int_2
    
    for (int i = 0; i < N; i++) {
        Int_2[i] = integrate2(x, i); 
    }
}

// Function to update vector x based on Int_1 and Int_2

void update_x(double *x, double *Int_1, double *Int_2) {
    int sign[N];//array of signs of partial gradients
    
    for (int i = 0; i < N; i++) {
        if (Int_1[i]-Int_2[i] >=DELTA/10.0) 
        {sign[i]=1;}
        else if(Int_1[i]-Int_2[i] <=-DELTA/10.0)
        {sign[i]=-1;}
        else {sign[i]=0;} 
        }
        for (int i=0; i<N; i++){
            x[i] -= DELTA * sign[i];  // Update x with the gradient step
            }
}// sign function

int main() {
    // Initialize vector x to zeros (x0)
    double x[N] = {1.2,-1.3};
    
    double Int_1[N], Int_2[N];
    
    // Compute initial values of Int_1 and Int_2
    
    compute_Int_1(x, Int_1);
    compute_Int_2(x, Int_2);

    // Iterative loop (while t <= T)
    for (int t = 0; t < T; t++) {
        // Update x based on Int_1 and Int_2
        update_x(x, Int_1, Int_2);

        // Recompute Int_1 and Int_2 at the new x
        
        compute_Int_1(x, Int_1);
        compute_Int_2(x, Int_2);

        // Output the updated x after each iteration
        
        printf("Iteration %d, x = [", t);
        for (int i = 0; i < N; i++) {
            printf("%f, ", x[i]);
        }
        
        // Output the updated partial derivatives after each iteration
        
        for (int i = 0; i < N; i++) {
            printf("%f, ", Int_1[i]-Int_2[i]);
        }
        printf("]\n");
     }
     return 0;
}
\end{lstlisting}

\section{zigzag1 in C codes (for LOG2)}\label{sec-zigzag1}
\begin{lstlisting}
    // Online C compiler to run C program online
#include <stdio.h>
#include <math.h>

// Bivariate function f(x, y)
double f(double x, double y) {
    double term1 = -log(pow(x - 0.5, 2) + pow(y - 0.5, 2) + 0.00001);
    double term2 = -log(pow(x + 0.5, 2) + pow(y + 0.5, 2) + 0.01);
    return fmax(0.0, term1 + term2); // Use fmax to implement max{0, *}
}

// Function to compute x_t and y_t, updating x or y based on parity of t
void compute_x_y_t(int N, int T, double DELTA, int t, double initial_x, 
double initial_y, double* x_t, double* y_t) {
    *x_t = initial_x; // Initialize x_t with the given initial_x
    *y_t = initial_y; // Initialize y_t with the given initial_y

    // Determine whether to update x or y based on the parity of t
    int is_even = (t % 2 == 0);
    double* var_to_update = is_even ? x_t : y_t; 
    // Pointer to the variable to update

    double Int1 = 0.0, Int2 = 0.0, Int3 = 0.0, Int = 0.0;

    // Calculate Int1
    for (int i = 1; i < T; ++i) {
        double term1, term2;
        if (is_even) {
            term1 = f(*x_t - i * DELTA, *y_t);
            term2 = f(*x_t + i * DELTA, *y_t);
        } else {
            term1 = f(*x_t, *y_t - i * DELTA);
            term2 = f(*x_t, *y_t + i * DELTA);
        }
        if (term1 != 0.0) Int1 += DELTA * term1;
        if (term2 != 0.0) Int1 -= DELTA * term2;
    }

    // Calculate Int2
    for (int i = -10; i <= 10; ++i) {
        double xi = i * DELTA / 10.0;
        double f_val;
        if (is_even) {
            f_val = f(*x_t + xi, *y_t);
        } else {
            f_val = f(*x_t, *y_t + xi);
        }
        Int2 += (i * DELTA * f_val) / 100.0;
    }

    Int = Int1 - Int2;

    // Iterative update
    for (int iter = 0; iter <= T; ++iter) {
        double prev_var = *var_to_update;
        *var_to_update -= DELTA * copysign(1.0, Int);

        // Update Int1
        Int1 = 0.0;
        for (int i = 1; i <= 10; ++i) {
            double xi, yi;
            if (is_even) {
                xi = *x_t + DELTA + i * DELTA / 10.0;
                yi = *y_t;
                double prev_xi = prev_var - DELTA + i * DELTA / 10.0;
                Int1 += copysign(1.0, Int) * DELTA / 10.0 * f(xi, *y_t) 
                + copysign(1.0, Int) * DELTA / 10.0 * f(prev_xi, *y_t);
            } else {
                xi = *x_t;
                yi = *y_t + DELTA + i * DELTA / 10.0;
                double prev_yi = prev_var - DELTA + i * DELTA / 10.0;
                Int1 += copysign(1.0, Int) * DELTA / 10.0 * f(*x_t, yi)
                + copysign(1.0, Int) * DELTA / 10.0 * f(*x_t, prev_yi);
            }
        }

        // Update Int3 and Int2
        Int3 = Int2;
        Int2 = 0.0;
        for (int i = -10; i <= 10; ++i) {
            double shifted_var = *var_to_update + i * DELTA / 10.0;
            double f_val;
            if (is_even) {
                f_val = f(shifted_var, *y_t);
            } else {
                f_val = f(*x_t, shifted_var);
            }
            Int2 += (i * DELTA * f_val) / 100.0;
        }

        // Update Int
        Int = Int - Int1 - Int2 + Int3;
    }
}

// Main function
int main() {
    int N = 3;          // Power (currently unused, extensible)
    int T = 200;        // Number of iterations
    double DELTA = 0.02; // Step size
    int K = 15;         // Number of cycles

    double x_t = 1;   // Initial x_t
    double y_t = -1;   // Initial y_t

    // Loop for t from 1 to K
    for (int t = 1; t <= K; ++t) {
        double temp_x, temp_y;
        compute_x_y_t(N, T, DELTA, t, x_t, y_t, &temp_x, &temp_y);
        x_t = temp_x;  // Update x_t to the current computation result
        y_t = temp_y;  // Update y_t to the current computation result
        printf("Final value after %d updates: x_t = %f, y_t = %f\n", t,
        x_t, y_t);
    }

    // Output the final value after K updates
    return 0;
}

\end{lstlisting}

\section{zigzag2 in C++ codes (for LOG2)}\label{sec-zigzag2}
\begin{lstlisting}
    #include <stdio.h>
#include <math.h>
#include <stdlib.h>
#include <time.h>

#define PI 3.14159265358979323846
#define DELTA1 0.02    // Fixed step size
#define NUM_DIRECTIONS 20  // Number of random directions
#define T_ITER 100         // Number of iterations for 1D search
#define K 10          // Total number of steps

// Define the objective function f(x, y, N)
double f(double x, double y, int N) {
    double term1 = -log(pow(x - 0.5, 2) + pow(y - 0.5, 2) + 0.00001);
    double term2 = -log(pow(x + 0.5, 2) + pow(y + 0.5, 2) + 0.01);
    return pow(fmax(0.0, term1 + term2), N);
}

// Generate a random unit direction (v_x, v_y)
void generate_random_direction(double* v_x, double* v_y) {
    double theta = 2 * PI * ((double)rand() / RAND_MAX);
    *v_x = cos(theta);
    *v_y = sin(theta);
}

// Compute the difference along direction v
double compute_difference(double x, double y, double v_x, double v_y, 
double r_t, int N) {
    double f_plus = f(x + r_t * v_x, y + r_t * v_y, N);
    double f_minus = f(x - r_t * v_x, y - r_t * v_y, N);
    return fabs(f_plus - f_minus);
}

// Perform 1D search along direction v, return optimal step size s
double compute_s_along_v(double x, double y, double v_x, double v_y, 
int N, int T_iter, double DELTA) {
    double s_t = 0.0;
    double Int1 = 0.0, Int2 = 0.0, Int3 = 0.0, Int = 0.0;

    auto f_s = [&](double s) { return f(x + s * v_x, y + s * v_y, N); };

    for (int i = 1; i < T_iter; ++i) {
        double term1 = f_s(s_t - i * DELTA);
        double term2 = f_s(s_t + i * DELTA);
        if (term1 != 0.0) Int1 += DELTA * term1;
        if (term2 != 0.0) Int1 -= DELTA * term2;
    }

    for (int i = -10; i <= 10; ++i) {
        double si = i * DELTA / 10.0;
        Int2 += (i * DELTA * f_s(s_t + si)) / 100.0;
    }

    Int = Int1 - Int2;

    for (int t = 0; t <= T_iter; ++t) {
        double y = s_t;
        s_t -= DELTA * copysign(1.0, Int);

        Int1 = 0.0;
        for (int i = 1; i <= 10; ++i) {
            double si = s_t + DELTA + i * DELTA / 10.0;
            double yi = y - DELTA + i * DELTA / 10.0;
            Int1 += copysign(1.0, Int) * DELTA / 10.0 * f_s(si) +
                    copysign(1.0, Int) * DELTA / 10.0 * f_s(yi);
        }

        Int3 = Int2;
        Int2 = 0.0;
        for (int i = -10; i <= 10; ++i) {
            double shifted_s = s_t + i * DELTA / 10.0;
            Int2 += (i * DELTA * f_s(shifted_s)) / 100.0;
        }

        Int = Int - Int1 - Int2 + Int3;
    }

    return s_t;
}

// Single-step update of (x_t, y_t)
void update_x_y_t(double* x_t, double* y_t, double r_t, int N, 
int T_iter, double DELTA) {
    double max_diff = -1.0;
    double best_v_x = 0.0, best_v_y = 0.0;

    for (int i = 0; i < NUM_DIRECTIONS; ++i) {
        double v_x, v_y;
        generate_random_direction(&v_x, &v_y);
        double diff = compute_difference(*x_t, *y_t, v_x, v_y, r_t, N);
        if (diff > max_diff) {
            max_diff = diff;
            best_v_x = v_x;
            best_v_y = v_y;
        }
    }

    double s = compute_s_along_v(*x_t, *y_t, best_v_x, best_v_y, 
    N, T_iter, DELTA);

    *x_t += s * best_v_x;
    *y_t += s * best_v_y;
}

// Main function
int main() {
    srand(time(NULL));

    int N = 3;
    double x_t = 0.1;
    double y_t = -1;
    double r_t = DELTA1;

    printf("Initial: x_t = %f, y_t = %f, r_t = %f, f(x_t, y_t, N) 
    = %f\n", x_t, y_t, r_t, f(x_t, y_t, N));

    for (int t = 1; t <= K; ++t) {
        update_x_y_t(&x_t, &y_t, r_t, N, T_ITER, DELTA1);
        r_t += DELTA1;
        printf("Step %d: x_t = %f, y_t = %f, r_t = %f, f(x_t, y_t, N) 
        = %f\n", t, x_t, y_t, r_t, f(x_t, y_t, N));
    }

    printf("Final value after %d steps: x_t = %f, y_t = %f, 
    r_t = %f, f(x_t, y_t, N) = %f\n", K, x_t, y_t, r_t, f(x_t, y_t, N));

    return 0;
}
\end{lstlisting}

\end{document}